\documentclass[12pt]{amsart} 
\usepackage[margin=1in]{geometry} 
\usepackage{amsthm,amsfonts,amsmath,amssymb} 
\usepackage{dsfont} 
\usepackage[T1]{fontenc} 
\usepackage[utf8]{inputenc} 
\usepackage{fourier} 

\theoremstyle{plain} 
\newtheorem{theorem}{Theorem} 
\newtheorem{lemma}[theorem]{Lemma} 

\DeclareMathOperator{\mre}{Re}

\begin{document}

\author{Andriy Bondarenko}
\address{Andriy Bondarenko\\
Department of Mathematical Sciences \\
Norwegian University of Science and Technology \\
NO-7491 Trondheim \\
Norway} \email{andriybond@gmail.com}

\author{Ole Fredrik Brevig} \address{Ole Fredrik Brevig \\
Department of Mathematical Sciences \\
Norwegian University of Science and Technology \\
NO-7491 Trondheim \\
Norway} \email{ole.brevig@ntnu.no}

\author{Eero Saksman} \address{Eero Saksman \\
Department of Mathematics and Statistics \\
University of Helsinki \\
FI-00170 Helsinki \\
Finland} \email{eero.saksman@helsinki.fi}

\author{Kristian Seip} \address{Kristian Seip\\Department of Mathematical Sciences \\
Norwegian University of Science and Technology \\
NO-7491 Trondheim \\
Norway} \email{kristian.seip@ntnu.no}

\author{Jing Zhao} \address{Jing Zhao\\Department of Mathematical Sciences \\
Norwegian University of Science and Technology \\
NO-7491 Trondheim \\
Norway} \email{jingzh95@me.com}

\thanks{The research of Bondarenko, Brevig, Seip, and Zhao is supported by Grant 227768 of the Research Council of Norway. The research of Saksman is supported by the Finnish Academy  Center of Excellence ``Analysis and Dynamics''.}

\title{Pseudomoments of the Riemann zeta function}


\begin{abstract}
The $2$kth pseudomoments of the Riemann zeta function $\zeta(s)$ are, following Conrey and Gamburd, the $2k$th integral moments of the partial sums of $\zeta(s)$ on the critical line. For fixed $k>1/2$, these moments are known to grow like $(\log N)^{k^2}$, where $N$ is the length of the partial sum, but the true order of magnitude remains unknown when $k\le 1/2$. We deduce new Hardy--Littlewood inequalities and apply one of them to improve on an earlier asymptotic estimate when $k\to\infty$. In the case $k<1/2$, we consider pseudomoments of $\zeta^{\alpha}(s)$ for $\alpha>1$ and the question of whether the lower bound $(\log N)^{k^2\alpha^2}$ known from earlier work yields the true growth rate. Using ideas from recent work of Harper, Nikeghbali, and Radziwi{\l\l} and some probabilistic estimates of Harper, we obtain the somewhat unexpected result that these pseudomements are bounded below by $\log N$ to a power larger than $k^2\alpha^2$ when $k<1/e$ and $N$ is sufficiently large.    
\end{abstract}

\maketitle

\section{Introduction}
An important and longstanding problem in the theory of the Riemann zeta function $\zeta(s)$ is to compute the moments
\[M_k(T):=\frac{1}{T}\int_T^{2T}|\zeta(1/2+it)|^{2k}dt \]
for large $T$ and all $k>0$. One expects that
\[ M_k(T) \sim A_k (\log T)^{k^2}\]
for some constant $A_k$ for which one even has precise predictions \cite{KS00}. This asymptotic behavior is only known to hold when $k=1,2$ by results of respectively Hardy and Littlewood \cite{HL16} and Ingham \cite{Ingham26}. An unconditional lower bound $M_k(T)\gg (\log T)^{k^2}$ is known in the range $k\ge 1$ \cite{RS13}, and this holds conditionally for all $k>0$ by work of Ramachandra (see~\cite{Ram1,Ram2}) and Heath-Brown \cite{HB81}. The optimal upper bound $M_k(T)\ll (\log T)^{k^2}$ has been established unconditionally for $k=1/n, 1+1/n$ ($n$ a positive integer) by results of respectively Heath-Brown \cite{HB81} and Bettin, Chandee, and Radziwi{\l\l} \cite{BCR}. Harper~\cite{Har1}, building and improving on work of Soundararajan~\cite{Sound1}, showed that the upper bound of optimal order holds conditionally for all $k>0$.

By the classical approximation
\[\zeta(\sigma+i t)= \sum_{n\le x} n^{-\sigma-it} - \frac{x^{1-\sigma-it}}{1-\sigma-it}+O\left(x^{-\sigma}\right),\]
which holds uniformly in the range $\sigma\ge \sigma_0>0$, $|t|\le x$ (see \cite[Thm.~4.11]{T}), the problem of computing $M_k(T)$ can be recast as the problem of computing
\[  \frac{1}{T}\int_T^{2T}|Z_N(it)|^{2k}dt, \] 
where $Z_N(s) := \sum_{n\leq N} n^{-1/2-s}$ and $N=\lfloor 2T\rfloor$.  
By the Bohr correspondence, which for every prime $p$ allows us to associate $p^{-it}$ with an independent Steinhaus random variable, we may think of the interval $[T,2T]$ as a subset of $\mathbb{T}^{\infty}$. An interesting question is then to understand the distribution of $Z_N(i t)$ for fixed $N$ on the entire torus $\mathbb{T}^{\infty}$ and, in particular, to compare with what we have on the subset $[T,2T]$. Following Conrey and Gamburd \cite{CG06}, we refer to the corresponding integral moments
\[\Psi_k(N):= \lim_{T\to\infty} \frac{1}{T} \int_{0}^{T} \left| Z_N(it)\right|^{2k} dt 
\]
for $k>0$ as the $2k$th pseudomoments of the Riemann zeta function $\zeta(s)$. 

Conrey and Gamburd found that 
\begin{equation}\label{eq:CG} 
	\Psi_k(N) = C_k (\log{N})^{k^2} + O\left((\log{N})^{k^2-1}\right) 
\end{equation}
when $k$ is an integer. Here $C_k=a_k \gamma_k$, with $a_k$ an arithmetic factor defined by
\[a_k := \prod_{p} \left(1-\frac{1}{p}\right)^{k^2}\sum_{j=0}^\infty \frac{\binom{j+k-1}{j}^2}{p^j}\]
and $\gamma_k$ the volume of a convex polytope. In the non-integer case, it is known from \cite{BHS15} that 
\begin{equation}\label{eq:gtr2} 
	\Psi_k(N) \asymp_k (\log{N})^{k^2}, \qquad k>1/2, 
\end{equation}
and that 
\begin{equation}\label{eq:low2} 
	\Psi_k(N) \gg_k (\log{N})^{k^2}, \qquad k>0. 
\end{equation}
In the range $k\le 1/2$, however, the results of \cite{BHS15} imply only that $\Psi_{1/2}(N)\ll (\log\log N)(\log N)^{1/4} $ and that
\[\Psi_k(N) \ll_k (\log{N})^{k/2}, \qquad 0<k<1/2.\]
The upper bounds of \cite{BHS15} were established by means of Helson's generalization of the M. Riesz theorem on the conjugation operator \cite{H1}, while the lower bounds were deduced from certain Hardy--Littlewood inequalities established in \cite{BHS15}. Recently, by replacing the use of Helson's theorem by a method involving an appropriate mollifying of $Z_N(s)$, Heap \cite{Heap} was able to get the much improved bound
\[ \Psi_k(N)\ll (\log N)^{\alpha_k} (\log\log N)^{1/2-\alpha_k}, \quad 0<k\le 1/2,\]
with $\alpha_k=k/(4(1-k))$.

The methods of \cite{BHS15} produced a lower bound of super-exponential decay and an upper bound of super-exponential growth for $\Psi_k(N)/(\log N)^{k^2}$ when $k\to \infty$. In view of \eqref{eq:CG}, and the well-known asymptotic expansion 
\[\log{a_k} = -k^2\log(2e^\gamma \log{k}) + O\left(\frac{k^2}{\log{k}}\right)\]
one suspects that super-exponential decay is correct, and this was conjectured in \cite[Sec.~5]{BHS15}. We will demonstrate that this is true, by replacing the estimates coming from Helson's theorem \cite{H1} with a new Hardy--Littlewood inequality. We also include additional details in the computation of the lower estimate from \cite{BHS15} to obtain an explicit lower bound for comparison.
\begin{theorem}\label{thm:zetapseudo} 
	Suppose that $k\ge 1$. Then 
	\begin{align*}
		\limsup_{N\to\infty}\frac{\Psi_k(N)}{(\log N)^{k^2}} &\leq \frac{1}{\Gamma(k+1)^{k}}\,\prod_{p}\left(1-\frac{1}{p}\right)^{k^2}\left(1-\frac{k}{\lfloor k\rfloor}\frac{1}{p}\right)^{-k\lfloor k\rfloor}, \\
		\liminf_{N\to\infty}\frac{\Psi_k(N)}{(\log N)^{k^2}} &\geq \frac{1}{\Gamma\left(\lfloor 2k\rfloor k+1\right)^{\frac{k}{\lfloor 2k\rfloor}}}\,\prod_{p} \left(1-\frac{1}{p}\right)^{k^2} \left(1 + \lfloor 2k \rfloor k \,\frac{1}{p}\right)^\frac{k}{\lfloor 2k \rfloor}. 
	\end{align*}
	In particular, as $k \to \infty$, we get that 
	\begin{equation}\label{eq:squeeze} 
		\exp\left(\left(-2 + o(1)\right)k^2\log k\right) \ll \frac{\Psi_k(N)}{(\log N)^{k^2}} \ll \exp\left(\left(-1 + o(1)\right)k^2\log k\right). 
	\end{equation}
\end{theorem}
It is interesting to observe the similarity between the lower bound in \eqref{eq:squeeze} and the unconditional bound
\[\frac{M_k(T)}{(\log T)^{k^2}}\gg \exp\left(\left(-2 + o(1)\right)k^2\log k\right)\]
obtained by Radziwi{\l\l} and Soundararajan \cite{RS13}. Likewise, we observe that the upper bound in \eqref{eq:squeeze} is in agreement with the expected behavior
\[\frac{M_k(T)}{(\log T)^{k^2}}\asymp \exp\left(\left(-1 + o(1)\right)k^2\log k\right)\]
conjectured by Keating and Snaith \cite{KS00} (see also \cite{CG01}). 

The range $0< k \le 1/2$ remains unsettled. In this paper, we consider the closely related problem of computing the pseudomoments of $\zeta^{\alpha}(s)$ for $\alpha>1$. The somewhat surprising conclusion is that, in this case, the lower bound obtained from the Hardy--Littlewood inequality of \cite{BHS15} does not give the right asymptotic order for small $k$. To state this result, we start from the Dirichlet series  
\[ \zeta^{\alpha}(s)=\sum_{n=1}^{\infty} d_{\alpha}(n) n^{-s} \]
and the corresponding partial sums $Z_{N,\alpha}(s):=\sum_{n\le N}d_{\alpha}(n) n^{-s-1/2}$, and 
define the pseudomoments of $\zeta^{\alpha}(s)$ as 
\[\Psi_{k,\alpha}(N):= \lim_{T\to \infty} \frac{1}{T} \int_0^T |Z_{N, \alpha}(i t)|^{2k} dt.\] 
We know\footnote{The techniques used in the proof of Theorem~\ref{thm:zetapseudo} can also be used to improve the estimates for the constants in \eqref{dream} and other examples from \cite{BHS15} when $k>1/2$. We omit the details.} from \cite{BHS15} that  
\begin{equation}\label{dream} 
	\Psi_{k,\alpha}(N) \asymp (\log N)^{k ^2 \alpha^2} 
\end{equation}
when $k>1/2$. However, for small $k$ this asymptotic relation fails.
\begin{theorem}\label{th1} 
	Suppose that $\alpha\ge 1$. For every $0<k<1/2$, there exists a positive constant $c(k,\alpha)$ such that
	\[\Psi_{k,\alpha}(N) \gg (\log N)^{k \log \alpha^2}\exp\left(-c(k,\alpha)\sqrt{\log\log N\log\log\log N}\right)\]
	holds for arbitrarily large $N$. 
\end{theorem}
This is incompatible with \eqref{dream} when $\alpha>1$ and $k<(\log \alpha^2)/\alpha^2$. Indeed, we observe that, whenever $k<1/e$, we can find $\alpha>1$ such that \eqref{dream} fails. Hence, while our results for $k>1/2$ shows a striking similarity between $\Psi_k(N)$ and the conjectured asymptotics of $M_k(N)$, Theorem~\ref{th1} reveals a different situation for small $k$. If we agree that the moments of $\zeta^{\alpha}(s)$ are just the moments of $|\zeta(s)|^{\alpha}$, then we may phrase this state of affairs in the following way: Theorem~\ref{th1} reveals that there is a discrepancy between the behavior of the pseudomoments and the moments of $\zeta^\alpha(s)$ for small $k$ when $\alpha>1$. (Here we compare with Heath-Brown's unconditional result for $k=1/n$ \cite{HB81} or Harper's conditional result for all $k>0$ \cite{Har1}.)

We may think of the problem of computing the pseudomoments $\Psi_k(N)$ in at least two different ways. From a functional analytic point of view, the underlying operator is that of partial sums acting on Hardy spaces of Dirichlet series: 
 \[ S_N\left(\sum_{n=1}^{\infty} a_n n^{-s}\right):=\sum_{n=1}^{N} a_n n^{-s}. \]
Following Bayart \cite{Bayart02}, we define the Hardy space $\mathcal{H}^q$ for $0<q<\infty$ by taking the closure of all Dirichlet polynomials $f$ with respect to the norm (or quasi-norm when $0<q<1$)
\[\|f\|_q:=\left(\lim_{T\to\infty} \frac{1}{T} \int_0^T |f(it)|^q dt  \right)^{1/q}.\]
The aforementioned theorem of Helson \cite{H1} shows that when $1<q<\infty$, there is a uniform bound on the norm of $S_N$ when it acts on $\mathcal{H}^q$. We refer to the forthcoming paper \cite{BBSS18} for some new estimates on the growth of the norm of $S_N$ in the interesting range $0<q\le 1$. The computation of the pseudomoments deals with the special situation when $S_N$ acts on functions with a strong multiplicative structure. Specifically, we may write
\begin{equation}\label{eq:finite} Z_N(s)=S_N \left( \prod_{p\le N} \left(1-p^{-1/2-s}\right)^{-1}\right); \end{equation}
the $\mathcal{H}^{2k}$ norm of the finite Euler product on which $S_N$ acts, can be estimated plainly, and hence the problem of computing $\Psi_k(N)$ and, in particular, the question of whether \eqref{eq:gtr2} holds, can be thought of as: How much does the ``additive'' operator $S_N$ distort the norm of the finite Euler product in \eqref{eq:finite}? Theorem~\ref{th1} indicates that the distortion becomes severe for small $k$.

Alternatively, we may think of our problem in probabilistic terms. We then associate with every prime $p$ a Steinhaus random variable $z(p)$, i.e., $z(p)$ is a random variable that is equidistributed on the unit circle. Assuming the random variables $z(p)$ to be independent, we define $z(n)$ by requiring it to be a completely multiplicative function for every point in our probability space. The relation to our problem of computing $\Psi_k(N)$ is given by the well-known norm identity
\[{\mathbb E} \left(\left|\sum_{n\le N} \frac{z(n)}{\sqrt{n}}\right|^q\right) =\lim_{T\to \infty} \frac{1}{T} \int_0^T \left|Z_N(it) \right|^q dt,\]
valid for all $q>0$ (see \cite[Sec.~3]{SS09}). Using this terminology, we may think of $Z_N$ as a sum of random multiplicative functions. A probabilistic approach to problems of computing integral moments, based on this viewpoint, can be found in recent work of Harper \cite{Har2,Harper} and Harper, Nikeghbali, and Radziwi{\l\l} \cite{HNR}. 

In many situations, both approaches apply equally well, but sometimes one viewpoint is more illuminating and profitable than the other, and occasionally it is useful to combine them.  
The reader may notice that we take the functional analytic point of view in our proof of Theorem~\ref{thm:zetapseudo}. Here the main point is to find appropriate substitutes for conventional results such as  Helson's version of the M. Riesz theorem and Riesz--Thorin interpolation, adapted to the multiplicative structure of our problem. The basic tool to be used, called Hardy--Littlewood inequalities, will be developed in the next section; the proof of Theorem~\ref{thm:zetapseudo} then follows in Section~\ref{se:euler}. The proof of Theorem~\ref{th1}, on the other hand, to be found Section~\ref{se:harper}, relies on the probabilistic approach and ideas found in \cite{Har2} and \cite{HNR}. 

\section{Hardy--Littlewood inequalities}\label{sec:HL}
The canonical example of the type of inequality we are interested in is Helson's inequality \cite{Helson06}, which states that
\begin{equation} \label{eq:Helson}
	\left(\sum_{n=1}^N \frac{|a_n|^2}{d(n)}\right)^\frac{1}{2} \leq \|f\|_1,
\end{equation}
for Dirichlet polynomials $f(s) = \sum_{n=1}^N a_n n^{-s}$. Here $d(n)$ denotes the divisor function. We may define the general divisor function $d_\alpha(n)$ for $\alpha\geq1$ by the rule 
\begin{equation}\label{eq:dadef} 
	\zeta^\alpha(s) = \sum_{n=1}^\infty d_\alpha(n) n^{-s}, \qquad \sigma>1, 
\end{equation}
which was tacitly assumed in the introduction. If $k$ is an integer, then it is clear that $d_k(n)$ denotes the number of ways we may write $n$ as a product of $k$ positive integers, since
\begin{equation} \label{eq:dkn}
	d_k(n) = \sum_{n_1\cdots n_k = n} 1.
\end{equation}
In particular, $d(n) = d_2(n)$. Another basic observation is that $d_\alpha(n)$ is a multiplicative function, which means that it is completely determined by its values at powers of the prime numbers. The Euler product formula for $\zeta^{\alpha}(s)$ shows that, in fact, 
\begin{equation}\label{eq:multext} 
	d_{\alpha}(p^j)=\binom{j+\alpha-1}{j}
\end{equation}
for every prime $p$ and every nonnegative integer $j$. The submultiplicative estimate 
\begin{equation} \label{eq:logconvex}
	d_\alpha(mn) \leq d_\alpha(m)d_\alpha(n)
\end{equation}
follows at once from \eqref{eq:multext}.
Our starting point is the following extension of Helson's inequality \eqref{eq:Helson}, which corresponds to the case $k=2$ in \eqref{eq:lowerk}.

\begin{lemma} \label{lem:integers}
	Let $f(s) = \sum_{n=1}^N a_n n^{-s}$ and let $k$ be a positive integer. Then
	\begin{align}
		\left(\sum_{n=1}^N \frac{|a_n|^2}{d_k(n)}\right)^\frac{1}{2} &\leq \|f\|_{2/k}, \label{eq:lowerk}\\
		\|f\|_{2k} &\leq \left(\sum_{n=1}^N |a_n|^2 d_k(n)\right)^\frac{1}{2}. \label{eq:upperk}
	\end{align}
\end{lemma}

\begin{proof}
	As noted in \cite[pp.~203--204]{BHS15}, the inequality \eqref{eq:lowerk} follows from \cite[Cor. 3.4]{Burbea87} and the argument used in \cite{BHS15,Helson06}. The inequality \eqref{eq:upperk} is an improved version of \cite[Lem.~8]{Seip13}, which we now deduce. Adopting the notational convention $a_n = 0$ if $n\geq N$, we expand to find that
	\begin{align*}
		\|f\|_{2k}^{2k} = \|f^k\|_2^2 &= \sum_{n=1}^{N^k} \left|\sum_{n_1\cdots n_k = n} a_{n_1}\cdots a_{n_k}\right|^2\leq \sum_{n=1}^{N^k} d_k(n) \sum_{n_1\cdots n_k = n} |a_{n_1}|^2 \cdots |a_{n_k}|^2
		\intertext{by the Cauchy--Schwarz inequality and \eqref{eq:dkn}. We then apply \eqref{eq:logconvex} and obtain}
	\| f\|_{2k}^{2k}	&\leq \sum_{n=1}^{N^k} \sum_{n_1\cdots n_k = n} d_k(n_1) |a_{n_1}|^2 \cdots d_k(n_k) |a_{n_k}|^2 = \left(\sum_{n=1}^N |a_n|^2 d_k(n)\right)^k. \qedhere
	\end{align*}
\end{proof}

It seems likely that Lemma~\ref{lem:integers} holds for any $k\geq1$, but as far as we know this is still an open problem. In \cite{BHS15,Seip13}, different techniques were used to circumvent this.\footnote{The Hardy--Littlewood inequalities in \cite{BHS15,Seip13} are stated with a weight of the form $[d(n)]^\beta$, where $d(n)=d_2(n)$ denotes the usual divisor function. The difference between $[d(n)]^\beta$ and $d_\alpha(n)$ is marginal, but we have found it more natural to use $d_\alpha(n)$.} Specifically, it is proved in \cite{BHS15} that \eqref{eq:lowerk} holds if we only consider square-free integers in the lower bound. Using the M\"{o}bius function $\mu(n)$, which is the multiplicative function that is $0$ if $n$ is not square-free and $-1$ at each prime number, the Hardy--Littlewood inequality of \cite{BHS15} can be written as 
\begin{equation}\label{eq:BHSineq} 
	\left(\sum_{n=1}^N |a_n|^2 \, \frac{|\mu(n)|}{d_{2/q}(n)}\right)^\frac{1}{2} \leq \|f\|_q,
\end{equation}
for $0<q\leq2$. In \cite{Seip13}, Riesz--Thorin interpolation between the integers $k$ in \eqref{eq:upperk} is used to the prove that
\[\|f\|_q \leq \left(\sum_{n=1}^N |a_n|^2 d_\alpha(n)\right)^\frac{1}{2},\]
for $q\geq2$, where $\alpha=\alpha(q)>q/2$ (unless $q/2$ is an integer).

Our novel approach is to interpolate between the results in Lemma~\ref{lem:integers} using instead a completely multiplicative weight. The interpolation will be facilitated by a version of Weissler's inequality \cite{Weissler80} for Dirichlet polynomials. For $f(s) = \sum_{n=1}^N a_n n^{-s}$, define
\[\mathcal{W}_\varrho f(s) := \sum_{n=1}^N \varrho^{\Omega(n)} a_n n^{-s}.\]
The following result is deduced in \cite[Sec.~3]{Bayart02}.
\begin{lemma} \label{lem:weissler}
	Suppose that $0<q_1\leq q_2 < \infty$ and let $0<\varrho \leq \sqrt{q_1/q_2}$. Then
	\[\|\mathcal{W}_\varrho f\|_{q_2} \leq \|f\|_{q_1},\]
	for every Dirichlet polynomial $f(s) = \sum_{n=1}^N a_n n^{-s}$.
\end{lemma}

Our replacement for $d_\alpha(n)$ will be the multiplicative function 
\begin{equation}\label{eq:multmixweight} 
	\Phi_\alpha(n) := d_{\lfloor \alpha \rfloor}(n)\, \left(\frac{\alpha}{\lfloor \alpha \rfloor}\right)^{\Omega(n)}, 
\end{equation}
where $\Omega(n)$ denotes the number of prime factors in $n$, counting multiplicity. Observe that $\Phi_\alpha(n)=d_\alpha(n)$ whenever $\alpha$ is an integer. Note also that $\Phi_\alpha(n) =d_\alpha(n)$ if $n$ is square-free. We will prove that $\Phi_\alpha(n)$ has the same average order as $d_\alpha(n)$, a fact that for our purposes makes it a satisfactory substitute. Only the bound for $q\ge 2$ in the following theorem will be used in the proof of Theorem~\ref{thm:zetapseudo}. Since the proofs are similar, and both bounds are of intrinsic interest, we have found it natural to treat the whole range $0<q<\infty$; the bound for $q\le 2$ will find applications in \cite{BBSS18}.

\begin{theorem}\label{thm:multmixedineqs} 
	If $f(s) = \sum_{n=1}^N a_n n^{-s}$, then 
	\begin{align}
		\left(\sum_{n=1}^N \frac{|a_n|^2}{\Phi_{2/q}(n)}\right)^\frac{1}{2} &\leq \|f\|_q, & q\leq 2, \label{eq:multmix2} \\
		\|f\|_q &\leq \left(\sum_{n=1}^N |a_n|^2 \Phi_{q/2}(n)\right)^\frac{1}{2}, & q\geq2. \label{eq:multmix1} 
	\end{align}
\end{theorem}
\begin{proof}
	We begin with \eqref{eq:multmix2}. In light of Lemma~\ref{lem:integers}, we may assume there is some positive integer $k$ such that \[\frac{2}{k+1} < q < \frac{2}{k}.\] In particular, $k = \lfloor 2/q \rfloor$. We therefore apply Lemma~\ref{lem:weissler} with
	\[\varrho = \sqrt{q/(2/k)} = \sqrt{\frac{q}{2}\left \lfloor \frac{2}{q} \right\rfloor}\]
	and \eqref{eq:lowerk} to obtain
	\[\|f\|_q \geq \|\mathcal{W}_\varrho f\|_{2/k} \geq \left(\sum_{n=1}^N \frac{|a_n|^2 \varrho^{2\Omega(n)}}{d_k(n)}\right)^\frac{1}{2} = \left(\sum_{n=1}^N \frac{|a_n|^2}{\Phi_{2/q}(n)}\right)^\frac{1}{2}.\]
	For the proof of \eqref{eq:multmix1}, we may assume that $2k < q < 2(k+1)$ since $q\geq2$. We use Lemma~\ref{lem:weissler} (in reverse) with $\varrho = \sqrt{2k/q}$ and \eqref{eq:upperk} to conclude that
	\[\|f\|_q \leq \|\mathcal{W}_\varrho^{-1} f\|_{2k} \leq \left(\sum_{n=1}^N |a_n|^2 \varrho^{-2\Omega(n)} d_k(n) \right)^\frac{1}{2} = \left(\sum_{n=1}^N |a_n|^2 \Phi_{q/2}(n)\right)^\frac{1}{2}. \qedhere\]
\end{proof}

From \eqref{eq:dadef} it follows by the Selberg--Delange method (see e.g.~\cite[Ch.~II.5]{Tenenbaum}) that the average order of $d_\alpha(n)$ is given by 
\begin{equation}\label{eq:avgdiv} 
	\frac{1}{N} \sum_{n=1}^N d_\alpha(n) = \frac{1}{\Gamma(\alpha)}(\log{N})^{\alpha-1} + O\left((\log{N})^{\alpha-2}\right). 
\end{equation}
We will now show that $\Phi_\alpha(n)$ has the same average order, up to a bounded factor. To that end, we consider the associated Dirichlet series and factor out a suitable power of $\zeta(s)$ from the Euler product, to obtain
\[\mathcal{F}_\alpha(s) := \sum_{n=1}^\infty \Phi_\alpha(n) n^{-s} = \zeta^\alpha(s) \prod_p \left(1-p^{-s}\right)^\alpha \left(\sum_{j=0}^\infty \Phi_\alpha(p^j)\, p^{-js}\right).\]
For $|z|<\lfloor \alpha \rfloor/\alpha$, it is now convenient to set 
\begin{equation}\label{eq:genfunc} 
	G_\alpha(z) := (1-z)^\alpha \sum_{j=0}^\infty \Phi_\alpha(p^j) z^j = (1-z)^\alpha\left(1-\frac{\alpha}{\lfloor \alpha \rfloor}z\right)^{-\lfloor \alpha \rfloor} ,
\end{equation}
so that $\mathcal{F}_\alpha(s)=\zeta^\alpha(s)\mathcal{G}_\alpha(s)$, where
\[\mathcal{G}_\alpha(s) := \prod_{p} G_\alpha(p^{-s}).\]
To prove the desired average order result for $\Phi_{\alpha}(n)$, we require the following simple estimates. 
\begin{lemma}\label{lem:Gaest} 
	If $\alpha\geq1$ and $0 \leq x < \lfloor \alpha \rfloor/\alpha$, then 
	\begin{equation}\label{eq:decreasing} 
		G_{\alpha+1}(x) \leq G_\alpha(x). 
	\end{equation}
	Moreover, $G_\alpha$ enjoys uniform estimates for $0\leq x \leq 1/2$,
	\[1 \leq G_\alpha(x) \leq 1 + x^2 
	\begin{cases}
		16(\alpha-1)/(2-\alpha)^3, & 1\leq \alpha < 2, \\
		384, & \alpha\geq 2. 
	\end{cases}
	\]
\end{lemma}
\begin{proof}
	To prove \eqref{eq:decreasing}, we look at the Taylor expansion of the logarithm
	\[\log\left( G_{\alpha}(x)\right) = \sum_{j=2}^\infty \frac{x^j}{j}\left(\lfloor \alpha \rfloor\left(\frac{\alpha}{\lfloor \alpha \rfloor}\right)^j-\alpha\right).\]
	It is sufficient to show that $C_j(\alpha+1) \leq C_j(\alpha)$, where
	\[C_j(\alpha) := \lfloor \alpha \rfloor\left(\frac{\alpha}{\lfloor \alpha \rfloor}\right)^j-\alpha.\]
	Clearly $C_j(\lfloor \alpha \rfloor)=C_j([\alpha+1])=0$. We set $\alpha=\lfloor \alpha \rfloor+t$ for $0\leq t < 1$, and differentiate to find that
	\[\frac{d}{dt}\,C_j(\alpha) = j\left(\frac{\alpha}{\lfloor \alpha \rfloor}\right)^{j-1} - 1 \,\geq\, j\left(\frac{\alpha+1}{\lfloor \alpha+1\rfloor}\right)^{j-1}-1 = \frac{d}{dt}\,C_j(\alpha+1).\]
	The lower bound in the second statement is just Bernoulli's inequality,
	\[\left(1-\frac{\alpha}{\lfloor \alpha \rfloor}x\right)^{\lfloor \alpha \rfloor/\alpha} \leq 1-x.\]
	The upper bounds can be computed with Taylor's theorem. By \eqref{eq:decreasing}, we only need to consider $1\leq \alpha <2$ and $\alpha\geq2$. The precise value of the constants are unimportant; we have obtained ours by rather coarse estimates. 
\end{proof}

From \eqref{eq:genfunc} and Lemma~\ref{lem:Gaest}, we get that the Dirichlet series representing $\mathcal{G}_\alpha(s)$ is absolutely convergent for
\[\mre s>\max\left(1/2,\,\log_2(\alpha/\lfloor \alpha \rfloor)\right).\]
Hence we apply the Selberg--Delange method (see e.g.~\cite[Ch.~II.5]{Tenenbaum}) to deduce that the average order of $\Phi_\alpha(n)$ is the same as the average order of $d_\alpha(n)$ given by \eqref{eq:avgdiv}. 
\begin{lemma}\label{lem:avgorder} 
	Let $\Phi_\alpha(n)$ denote the weight \eqref{eq:multmixweight} for fixed $\alpha\geq1$. Then
	\[\frac{1}{x} \sum_{n \leq x} \Phi_\alpha(n) = \frac{\mathcal{G}_\alpha(1)}{\Gamma(\alpha)}(\log{x})^{\alpha-1} + O\left((\log{x})^{\alpha-2}\right).\]
\end{lemma}
Theorem~\ref{thm:multmixedineqs} and Lemma~\ref{lem:avgorder} have applications in the theory of Hardy spaces of Dirichlet series, as will be exhibited in the forthcoming paper \cite{BBSS18}.
 
\section{Proof of Theorem~\ref{thm:zetapseudo}}\label{se:euler}
\begin{proof}
	[Proof of the upper estimate in Theorem~\ref{thm:zetapseudo}] Inserting $Z_N$ into \eqref{eq:multmix1}, we get
	\[\Psi_k(N)=\|Z_N\|_{2k}^{2k} \leq \left(\sum_{n=1}^N \frac{d_{\lfloor k \rfloor}(n)}{n}\,\left(\frac{k}{\lfloor k\rfloor}\right)^{\Omega(n)}\right)^k.\]
	Using Lemma~\ref{lem:avgorder} and Abel summation, we find that
	\[\sum_{n=1}^N \frac{d_{\lfloor k\rfloor}(n)}{n}\,\left(\frac{k}{\lfloor k\ \rfloor}\right)^{\Omega(n)} = \frac{\mathcal{G}_{k}(1)}{\Gamma(k+1)}(\log{N})^{k} + O\left((\log{N})^{k-1}\right).\]
	We complete the proof by inspecting the Euler product for $\mathcal{G}_k(1)$ and \eqref{eq:genfunc}. For the asymptotic estimate, we may safely assume $k\geq 2$, in which case Lemma~\ref{lem:Gaest} gives $\mathcal{G}_{k}(1)\asymp 1$. Hence the main contribution to the decay comes from the Gamma function, and the desired result follows from Stirling's formula:
	\[\Gamma(k+1)^{k} = \exp\left((1+o(1))\,k^2\log k\right). \qedhere\]
\end{proof}

The following argument can be extracted from \cite[pp.~201--202]{BHS15}, but we include some details here for the reader's benefit. 
\begin{proof}
	[Proof of the lower estimate in Theorem~\ref{thm:zetapseudo}] We want to use \eqref{eq:BHSineq}, but $k=q/2\ge 1$. To remedy this, we write $2k=\ell r$ where $\ell \geq \lfloor 2k \rfloor$ is an integer to be chosen later that ensures that $r<2$. Note that if $n \leq N$, then
	\[\frac{|\mu(n)|}{d_{2/r}(n)}\,\left|\sum_{\substack{n_1\cdots n_\ell = n \\
	n_1,\ldots,n_\ell \leq N} } \frac{1}{\sqrt{n_1}}\cdots \frac{1}{\sqrt{n_\ell}}\right|^2 = \frac{|\mu(n)|}{d_{2/r}(n)}\,\frac{d^2_\ell(n)}{n} = \frac{|\mu(n)|}{n}\,d_{\ell k}(n).\]
	Using \eqref{eq:BHSineq} and removing all terms in the sum for which $N < n \leq N^\ell$, we get the lower bound
	\[\|Z_N\|_{2k}^{2k} = \|Z_N^\ell\|_r^r \geq \left(\sum_{n=1}^N \frac{|\mu(n)|}{n} d_{\ell k}(n)\right)^\frac{k}{\ell}.\]
	As above, one checks that
	\[\sum_{n=1}^N \frac{|\mu(n)|}{n} d_{\ell k}(n) = \widetilde{C}_k (\log{N})^{\ell k} + O\left((\log{N})^{\ell k-1}\right)\]
	with 
	\begin{equation}\label{eq:Ctilde} 
		\widetilde{C}_k = \frac{1}{\Gamma(\ell k+1)}\prod_{p} \left(1-\frac{1}{p}\right)^{\ell k}\left(1+\frac{\ell k}{p}\right). 
	\end{equation}
	The asymptotic behavior of the Euler product in \eqref{eq:Ctilde} has been estimated in \cite[p.~202]{BHS15}, where it was found that
	\[\prod_{p} \left(1-\frac{1}{p}\right)^{\ell k}\left(1+\frac{\ell k}{p}\right) = \exp\left(\left(-1+o(1)\right)\ell k\log\log(\ell k)\right).\]
	Therefore the decay is again controlled by $\Gamma(\ell k+1)^{k/\ell}$. Clearly, choosing $\ell$ as small as possible is optimal, and we therefore set $\ell=\lfloor 2k\rfloor$. The proof is completed by similar considerations as in the preceding argument. 
\end{proof}

\section{Proof of Theorem~\ref{th1}}\label{se:harper}
We prepare for the proof of Theorem~\ref{th1} by establishing two lemmas. 
\begin{lemma}
	Suppose that $\alpha\ge 1$. Then
	\[\mathbb{E}\left|\sum_{M/2< n\le M} d_{\alpha}(n) \alpha^{-\Omega(n)} {z(n)}n^{-1/2}\right|\gg(\log\log M)^{-3+o(1)},\qquad\]
	where the $o(1)$ term depends only on $M$. 
\label{lemma1} \end{lemma}
Here we applied the probabilistic notation of the introduction. We defer the proof of Lemma~\ref{lemma1} until the end of this subsection. Our second lemma is a result on the distribution of
\[\mathcal{N}(x,m):=\sum_{\substack{n\le x \\ \Omega(n)=m}}1,\]
similar in spirit to the Erd\H{o}s--Kac theorem, saying that $\mathcal{N}(x,m)$ is mainly concentrated on
\[I_C:=\left[\log\log x-C\sqrt{\log\log x\log\log\log x},\,\log\log x+C\sqrt{\log\log x\log\log\log x}\right]\]
when $x$ is large and $C$ is a suitable positive constant. To deduce this result, we rely on an estimate of Sathe (see \cite{Selberg}) saying that 
\begin{equation}\label{Sat} 
	\mathcal{N}(x,m)\ll \frac{x}{\log x}\frac{(\log\log x)^{m-1}}{(m-1)!} 
\end{equation}
whenever $x>10$ and $1\le m\le (3/2)\log\log x$. Now suppose $\xi$ is a fixed positive number. Then choosing $C$ large enough and using Stirling's formula in \eqref{Sat}, we find that 
\begin{equation}\label{Sat1} 
	\sum_{\substack{m\le (3/2)\log\log x \\ m\not\in I_C}}\mathcal{N}(x,m)\le\frac{x}{2(\log\log x)^\xi} 
\end{equation}
when $x$ is sufficiently large. Using instead of~\eqref{Sat} the main result of \cite{BDN}, we deduce that 
\begin{equation}\label{Sat2} 
	\sum_{m\ge (3/2)\log\log x}\mathcal{N}(x,m)\le\frac{x}{(\log x)^{1/100}} 
\end{equation}
for $x$ large enough. Combining~\eqref{Sat1} and~\eqref{Sat2}, we obtain the following. 
\begin{lemma}\label{lem:EK} 
	Suppose $\xi$ is a given positive number. Then there exists a constant $C>0$ such that
	\[\sum_{m\not\in I_C}\mathcal{N}(x,m)\le\frac{x}{(\log\log x)^\xi}\]
	for all sufficiently large $x$. 
\end{lemma}
We also require the following result, which is \cite[Lem.~3]{BS}.
\begin{lemma} \label{lem:Pm}
	For $m\geq0$, define
	\[P_m\left(\sum_{n=1}^\infty a_n n^{-s}\right) := \sum_{\Omega(n)=m} a_n n^{-s}.\]
	Let $0<q<1$. Then $\|P_m f\|_q \ll m^{1/q-1} \|f\|_q$ for every Dirichlet polynomial $f$.
\end{lemma}

\begin{proof}
	[Proof of Theorem~\ref{th1}] We write
	\[D_{N,\alpha}(s):=\sum_{N/2< n\le N} d_\alpha(n)\alpha^{-\Omega(n)} n^{-s-1/2}\]
	so that
	\[Z_{N,\alpha}(s)-Z_{N/2,\alpha}(s)=\sum_{m\ge 0} {\alpha}^m P_mD_{N,\alpha}(s).\]
	By Lemma~\ref{lem:Pm}, we have for every $m$ and $0<q<1$ 
	\begin{equation}\label{main} 
		\| Z_{N,\alpha}-Z_{N/2,\alpha} \|_{q} \gg {\alpha}^{m}m^{1-1/q}\| P_m D_{N,\alpha} \|_{q}. 
	\end{equation}
	We will combine \eqref{main} with an estimate that we obtain from the two lemmas above. 
	
	In what follows, we will use that the $L^2$ norm of $D_{N,\alpha}$ can be estimated in a trivial way because $d_{\alpha}(n){\alpha}^{-\Omega(n)}\le 1$. First, applying H\"{o}lder's inequality in the form
	\[\| f\|_1^{2-q} \le \| f\|_q^{q} \|f\|_2^{2-2q}\]
	along with Lemma~\ref{lemma1} and a trivial $L^2$ estimate, we find that
	\[\left\|\sum_{m\geq0} P_mD_{N,\alpha} \right\|_q^q\gg(\log\log N)^{-6+o(1)}\]
	whenever $0<q<1$. Using the triangle inequality for the $L^q$ quasi-norm and the trivial bound $\| f\|_q\le \| f\|_2$, we obtain from this that
	\[\sum_{m\in I_C}\left\| P_mD_{N,\alpha} \right\|_{q}^q+\left\| \sum_{m\not\in I_C}P_mD_{N,\alpha} \right\|_{2}^q\gg(\log\log N)^{-6+o(1)}.\]
	Hence, by a trivial $L^2$ bound and an application of Lemma~\ref{lem:EK} with $\xi=16/q$, there exists a constant $C$ such that
	\[\sum_{m\in I_C}\left\| P_mD_{N,\alpha} \right\|_{q}^q \gg(\log\log N)^{-6+o(1)}.\]
	Thus, since $|I_C|=O\left(\sqrt{\log\log N\log\log\log N}\right)$, there exists an $m$ satisfying
	\[\log\log N-C\sqrt{\log\log N\log\log\log N}\le m \le \log\log N+C\sqrt{\log\log N\log\log\log N}\]
	such that 
	\begin{equation}\label{eq:keyq} 
		\| P_mD_{N,\alpha} \|_{q}^q\ge (\log\log N)^{-6.5+o(1)}. 
	\end{equation}
	We now set $q=2k$. Combining~\eqref{main} and \eqref{eq:keyq}, we find that
	\[\| Z_{N,\alpha}-Z_{N/2,\alpha} \|_{2k}^{2k}\gg(\log N)^{k \log \alpha^2}\exp\left(-c(k,\alpha) \sqrt{\log\log N\log\log\log N}\right)\]
	for some positive constant $c(k,\alpha)$. Since
	\[\| Z_{N,\alpha}-Z_{N/2,\alpha} \|_{2k}^{2k}\le \| Z_{N,\alpha}\|_{2k}^{2k}+\|Z_{N/2,\alpha} \|_{2k}^{2k},\]
	this means that at least one of the pseudomoments $\Psi_{k,\alpha}(N/2)$ or $\Psi_{k,\alpha}(N)$ satisfies the lower bound asserted by the theorem. 
\end{proof}

In the following proof, we have adapted the method of \cite[Sec.~2]{HNR} which relies crucially on Harper's work \cite{Harper}. We refer to \cite[pp.~150--152]{HNR} for an illuminating outline of the method.

\begin{proof}
	[Proof of Lemma~\ref{lemma1}] Let $\mathcal{S}_x$ be the set of $x$-smooth numbers, i.e.,
	\[\mathcal{S}_x:=\left\{n\in \mathbb{N}: \text{$p$ a prime such that $p|n$} \implies p\le x\right\}.\]
	We start with the following identity which holds for every real $t$:
	\begin{align}\label{mellin}
		\begin{split}
			&\int_{1}^{\infty} \sum_{\substack{y/2< n \le y \\ n\in \mathcal{S}_x}} d_{\alpha}(n) \alpha^{-\Omega(n)} z(n)n^{-1/2}\,\frac{dy}{y^{1 + 1/\log x + it }} \\ 
			&\qquad\qquad\qquad\qquad\qquad\qquad= \left(\frac{1-2^{-1/\log x-it}}{1/\log x+it} \right) \sum_{n\in \mathcal{S}_x}d_{\alpha}(n) \alpha^{-\Omega(n)} z(n)n^{-1/2 - 1/\log x- it }.
		\end{split}
	\end{align}
	Our first goal is to estimate the supremum of the right hand side in~\eqref{mellin} for $t$ from a reasonably short interval. We have
	\begin{equation} \label{exp}
		\begin{split}
			&\left|\sum_{n\in \mathcal{S}_x}d_{\alpha}(n)\alpha^{-\Omega(n)}z(n)  n^{-1/2 - 1/\log x- it}\right| = \prod_{p\le x}\left|1+\sum_{j=1}^\infty d_{\alpha}(p^j)\alpha^{-j} z(p)^j p^{-j(1/2+1/\log x +it)} \right| \\
			&\qquad\qquad\asymp\exp\left(\mre\left(\sum_{p\le x}z(p)p^{-1/2 - 1/\log x- it}\right)+\frac{1}{2\alpha}\mre\left(\sum_{p\le x}z(p)^2p^{-1 - 2/\log x- 2it}\right)\right)
		\end{split}
	\end{equation}
	for all points of the configuration space $(z(p))_{p\le x}$. As in \cite[Lem.~1]{HNR}, we can modify the proof of \cite[Cor.~2]{Harper} to show that 
	\begin{align*}
		\sup_{\substack{1\le t\le 2(\log\log x)^2 \\ |1-2^{-it}|\ge 1/4}} \left(\mre\left(\sum_{p\le x}z(p)p^{-1/2 - 1/\log x- it}\right)\right.&+\left.\frac {1}{2\alpha} \mre\left(\sum_{p\le x}z(p)^2p^{-1 - 2/\log x- 2it}\right)\right) \\
		& \ge \log\log x-\log\log\log x+ O\left((\log\log\log x)^{3/4}\right) 
	\end{align*}
	with probability $1-o(1)$ as $x\to\infty$. To achieve this, we add a minor technical detail: In the part of the argument that follows \cite[Sec.~6]{Harper}, we only take into account integers $1\le n \le (\log\log x)^2$, such that
	\[\min_{2n+1\le t \le 2n+2} |1-2^{-i t}| \ge 1/4,\]
	noting that the number of such $n$ is bounded below by a constant times $(\log\log x)^2$. Combining the latter inequality with~\eqref{exp}, we obtain that with probability $1-o(1)$
	\[\sup_{\substack{1\le t\le 2(\log\log x)^2 \\ |1-2^{-it}|\ge 1/4}}\left|\sum_{n\in \mathcal{S}_x}d_{\alpha}(n)\alpha^{-\Omega(n)}z(n)n^{-1/2 - 1/\log x- it}\right|\ge \frac{\log x}{(\log\log x)^{1+o(1)}}.\]
	Now taking the supremum of the absolute value of both sides in~\eqref{mellin}, we find that
	\[\int_{1}^{\infty}\left|\sum_{\substack{y/2< n \le y \\ n\in \mathcal{S}_x}}d_{\alpha}(n)\alpha^{-\Omega(n)}z(n)n^{-1/2}\right|\,\frac{dy}{y^{1 + 1/\log x}}\ge \frac{\log x}{(\log\log x)^{3+o(1)}}\]
	with probability $1-o(1)$. Hence taking the expectation over the entire configuration space $(z(p))_{p\le x}$, we finally obtain that, say for all $x>3$, 
	\begin{equation}\label{int} 
		\int_{1}^{\infty}\mathbb{E}\left|\sum_{\substack{y/2< n \le y \\ n\in \mathcal{S}_x}}d_{\alpha}(n)\alpha^{-\Omega(n)}z(n)n^{-1/2}\right|\,\frac{dy}{y^{1 + 1/\log x}}\ge {\frac{\log x}{(\log\log x)^{3+o(1)}}}. 
	\end{equation}
	Now we will show that the assertion of the lemma follows from~\eqref{int}. To this end, we begin by fixing a positive integer $M$. We will use~ \eqref{int} for $x$ such that $M=x^{10\log\log\log x}$. Applying the Cauchy--Schwarz inequality in the form $(\mathbb{E}|X|)^2\le \mathbb{E}|X|^2$ and recalling that $d_{\alpha}(n)\alpha^{-\Omega(n)}\le 1$, we find that
	\[\int_{\sqrt{M}}^{\infty}\mathbb{E}\left|\sum_{{\substack{y/2< n \le y \\ n\in \mathcal{S}_x}}}d_{\alpha}(n)\alpha^{-\Omega(n)}z(n)n^{-1/2}\right|\frac{dy}{y^{1 + 1/\log x}}\, \le  \int_{\sqrt{M}}^{\infty}\frac{1}{y^{1 + 1/\log x}} dy={\frac{\log x}{(\log\log x)^{5}}}.\]
	Combining this bound with~\eqref{int}, we find that 
	\begin{equation}\label{int1} 
		\int_{1}^{\sqrt{M}}\mathbb{E}\left|\sum_{{\substack{y/2< n \le y \\ n\in \mathcal{S}_x}}}d_{\alpha}(n)\alpha^{-\Omega(n)}z(n)n^{-1/2}\right|\,\frac{dy}{y^{1 + 1/\log x}}\ge {\frac{\log x}{(\log\log x)^{3+o(1)}}}, 
	\end{equation}
	which is the relation to be used below.
	
	Set $S_{M,\alpha}(z):=\sum_{M/2< n\le M}d_{\alpha}(n) \alpha^{-\Omega(n)} z(n)n^{-1/2}$, and let $\mathcal{S}_x^{\perp}$ be the set of integers with prime divisors that are all larger than $x$. Set
	\[\mathcal{R}_{M,x} := \left\{n_2 \in \mathcal{S}_x^{\perp}\,:\, \frac{M}{2} \leq n_1 n_2 \leq M \text{ for some } n_1 \in \mathcal{S}_x\right\}\]
	and decompose
	\[S_{M,\alpha}(z)=\sum_{n_2 \in \mathcal{R}_{M,x}} d_{\alpha}(n_2)\alpha^{-\Omega(n_2)}z(n_2)n_2^{-1/2} c_{n_2}\]
	where
	\[c_{n_2} := \sum_{\substack{M/(2 n_2) \leq n_1 \leq M/n_2 \\ n_1 \in \mathcal{S}_x}} d_{\alpha}(n_1)\alpha^{-\Omega(n_1)}z(n_1)n_1^{-1/2}.\]
	By using Helson's inequality \eqref{eq:Helson} with respect to the variables $z(n_2)$ for $n_2$ in $\mathcal{R}_{M,x}$, we find that
	\begin{equation} \label{E}
		\mathbb{E}|S_{M,\alpha}| \geq \left(\sum_{n_2 \in \mathcal{R}_{m,x}}\frac{|c_{n_2}|^2 d_{\alpha}(n_2)^2 \alpha^{-2\Omega(n_2)}}{d(n_2) n_2}\right)^\frac{1}{2} \geq \left(\sum_{x<p\le M} \frac{|c_p|^2}{2p}\right)^{1/2}.
	\end{equation}
	We now want to relate the right-hand side of \eqref{E} to the integral
	\[\int_{x}^{M} \left| \sum_{{\substack{M/(2y)\le n \le M/y \\ n\in \mathcal{S}_x}}} d_{\alpha}(n) \alpha^{-\Omega(n)} z(n)n^{-1/2} \right|^2 \frac{dy}{y} =\int_{x}^{M} |c_y|^2 \frac{dy}{y}.\]
	To this end, we begin by considering a short interval $[\xi, \xi+\xi^{\delta}] \subset [x,M]$, where $7/12<\delta<1$ is a fixed parameter. If $\xi$ is sufficiently large, then by \cite{HB}, this interval contains at least $\xi^{\delta}/(2\log \xi)$ primes. We partition accordingly the interval into $\lfloor\xi^{\delta}/(2\log \xi)\rfloor$ subintervals of equal length $\xi^{\delta}/ \lfloor\xi^{\delta}/(2\log \xi)\rfloor$. We make a one-to-one correspondence between these subintervals and the first $\lfloor\xi^{\delta}/(2\log \xi)\rfloor$ primes in $[\xi, \xi+\xi^{\delta}]$, and hence we associate with every $y$ in $[\xi, \xi+\xi^{\delta}]$ a prime $p=p(y)$ that is also in $[\xi, \xi+\xi^{\delta}]$. We write $\tilde{c}_y:=c_y-c_{p(y)}$ and notice that
	\[|c_y|^2\le 2 \left(|c_{p(y)}|^2+ |\tilde{c}_y|^2\right).\]
	A trivial estimate shows that 
	\begin{equation} \label{eq:short} 
		\mathbb{E}| \tilde{c}_y|^2 \ll \max \left(y^{\delta-1}, \frac{y}{M}\right). 
	\end{equation}  Using this construction, we get that
	\begin{align*} \int_{\xi}^{\xi+\xi^{\delta}} \left| \sum_{{\substack{M/(2y) \le n \le M/y \\ n\in \mathcal{S}_x}}} d_{\alpha}(n) \alpha^{-\Omega(n)} z(n)n^{-1/2} \right|^2 \frac{dy}{y} & \ll \log \xi \sum_{\xi\le p \le \xi+\xi^{\delta}} \frac{|c_p|^2}{p} + \int_{\xi}^{\xi+\xi^{\delta}} \frac{|\tilde{c}_y|^2}{y} dy \\ & \le \log M \sum_{\xi\le p \le \xi+\xi^{\delta}} \frac{|c_p|^2}{p} + \int_{\xi}^{\xi+\xi^{\delta}} \frac{|\tilde{c}_y|^2}{y} dy. \end{align*}
	Repeating this construction and summing over a suitable collection of intervals $[\xi,\xi+\xi^{\delta}]$, we then obtain
	\[\sum_{x<p\le M} \frac{|c_p|^2}p + \frac{1}{\log M}\int_{x}^{M} \frac{|\tilde{c}_y|^2}{y} dy \gg \frac{1}{\log M}\int_{x}^M \left| \sum_{{\substack{M/(2y)\le n \le M/y \\ n\in \mathcal{S}_x}}} d_{\alpha}(n) \alpha^{-\Omega(n)}z(n)n^{-1/2} \right|^2 \frac{dy}{y}.\]
	By the change of variables $u=M/y$ in the integral on the right-hand side and using that $\log M=10\log x \log\log\log x$, we now deduce that
	\begin{equation} \label{estimate}
		\begin{split}
			&\left(\sum_{x<p\le M}\frac{|c_p|^2}{p} +\frac{1}{\log M} \int_{x}^{M} \frac{|\tilde{c}_y|^2}{y} dy \right)^{1/2} \\
			&\qquad\qquad\qquad \gg \left(\frac{1}{\log x \log\log\log x} \int_{1}^{M/x} \left| \sum_{{\substack{u/2\le n \le u \\ n\in \mathcal{S}_x}}} d_{\alpha}(n) \alpha^{-\Omega(n)} z(n)n^{-1/2} \right|^2 \frac{du}{u}\right)^{1/2}.
		\end{split}		
	\end{equation}	
	We are now ready to finish the proof by putting our three basic estimates \eqref{int1}, \eqref{E}, and \eqref{estimate} together. First, by the Cauchy--Schwarz inequality, we have 
	\begin{align*}
		&\left(\int_{1}^{M/x} \left| \sum_{{\substack{u/2\le n \le u \\ n\in \mathcal{S}_x}}} d_{\alpha}(n) \alpha^{-\Omega(n)}z(n)n^{-1/2} \right|^2 \frac{du}{u}\right)^{1/2} \left(\int_{1}^{M/x}\frac{du}{u^{1+2/\log x}}\right)^{1/2} \\
		&\qquad\qquad\qquad\qquad\qquad\qquad\qquad\qquad \ge\int_{1}^{\sqrt{M}}\left|\sum_{{\substack{u/2\le n \le u \\ n\in \mathcal{S}_x}}}d_{\alpha}(n) \alpha^{-\Omega(n)} z(n)n^{-1/2}\right|\,\frac{du}{u^{1 + 1/\log x}}. 
	\end{align*}
	Therefore, taking expectation in~\eqref{estimate} and applying~\eqref{E} together with~\eqref{int1} and \eqref{eq:short}, we find that 
	\begin{align*}
		\mathbb{E}|S_M| & \gg \mathbb{E}\left|\left( \sum_{x<p\le M} \frac{|c_p|^2}{p}+ \frac{1}{\log M}\int_{x}^{M} \frac{|\tilde{c}_y|^2}{y} dy\right)^{1/2}\right|-\mathbb{E}\left|\left(\frac{1}{\log M} \int_{x}^{M} \frac{|\tilde{c}_y|^2}{y} dy \right)^{1/2}\right| \\
		& \gg\mathbb{E} \left|\left( \sum_{x<p\le M} \frac{|c_p|^2}{p}+\frac{1}{\log M}\int_{x}^{M} \frac{|\tilde{c}_y|^2}{y} dy \right)^{1/2}\right|-O\left((\log M)^{-1/2}\right) \\
		& \gg \frac{1}{\log x(\log\log\log x)^{1/2}} \int_{1}^{\sqrt{M}}\mathbb{E}\left|\sum_{{\substack{u/2\le n \le u \\ n\in \mathcal{S}_x}}}d_{\alpha}(n) \alpha^{-\Omega(n)} z(n)n^{-1/2}\right|\,\frac{du}{u^{1 + 1/\log x}} - O\left((\log M)^{-1/2}\right)\\
		& \ge (\log\log x)^{-3+o(1)}\ge(\log\log M)^{-3+o(1)}, 
	\end{align*}
	and hence the desired estimate has been established. 
\end{proof}

\bibliographystyle{amsplain} 
\bibliography{hardyp}

\end{document}